\documentclass[12pt]{article}

\usepackage{amsmath, amsthm, amssymb}
\usepackage{mathrsfs}
\usepackage{hyperref}
\usepackage{geometry}
\usepackage{enumerate}
\hypersetup{
    colorlinks=true,
    linkcolor=blue,
    citecolor=blue,
    urlcolor=blue
}

\newtheorem{theorem}{Theorem}[section]

\newtheorem{proposition}[theorem]{Proposition}

\theoremstyle{definition}
\newtheorem{definition}[theorem]{Definition}
\newtheorem{example}[theorem]{Example}

\theoremstyle{remark}

\newcommand{\A}{\mathcal{A}}
\newcommand{\doi}[1]{\href{https://doi.org/#1}{#1}}

\title{Structural Properties of $\varphi$-Semi Contractible Banach Algebras}
\author{
Hamid Sadeghi Nahrekhalaji\thanks{Email: h.sadeghi9762@gmail.com, sadeghi.9762@iau.ac.ir} \\
Department of Mathematics, Fari.C., Islamic Azad University, Isfahan, Iran
\and
Ali Ghafarpanah\thanks{Email: ghafarpanah@kazerunsfu.ac.ir, ghafarpanah2002@gmail.com} \\
Salman Farsi University of Kazerun, Kazerun, Iran
}
\date{}

\begin{document}

\maketitle

\begin{abstract}
In this paper, we introduce and investigate the notion of $\varphi$-semi contractibility for Banach algebras. We establish a complete characterization of $\varphi$-semi contractibility through the existence of specific elements satisfying particular algebraic conditions. We thoroughly examine the hereditary properties of this concept, including its behavior under closed ideals, continuous homomorphisms with dense range, and projective tensor products. Notably, we prove that the projective tensor product $\mathcal{A}\widehat{\otimes}\mathcal{B}$ is $\varphi\otimes\psi$-semi contractible precisely when both $\mathcal{A}$ and $\mathcal{B}$ are semi contractible with respect to their respective characters, where $\varphi \in \Delta(\mathcal{A})$ and $\psi \in \Delta(\mathcal{B})$. Furthermore, we introduce and analyze the approximate version of $\varphi$-semi contractibility, providing equivalent characterizations and investigating its stability under tensor products. \vspace{0.2 cm} 
\end{abstract}
\maketitle
{\footnotesize\noindent{\it Article history:}\\
Received: Month x, year\\
Received in revised form: Month x, year\\
Accepted: Month x, year}

\section{Introduction and preliminaries}
The study of amenability-like properties for Banach algebras relative to a character $\varphi$ has produced a rich family of notions (e.g.\ $\varphi$-amenability, character-amenability, $\varphi$-contractibility) that sit strictly between classical amenability and trivial one-dimensional behavior. 
These ideas were developed and popularized in a sequence of works; in particular the systematic study of $\varphi$-amenability was initiated by Kaniuth--Lau--Pym \cite{KLP1, KLP2}. 
In fact, they generalized the left amenability of Lau algebras \cite{L} for Banach algebras in general. Related forms of character amenability and contractibility were investigated by Monfared \cite{M}, Hu--Monfared--Traynor \cite{HMT} and others. See also \cite{Az}. 
The definitions and permanence results we give here are aligned with that literature and are intended to be compatible with standard cohomological language. 

Our objective is to define a weakening  of $\varphi$-contractibility that still yields tractable structural consequences and is easy to verify in many examples. 
The basic idea is to replace the requirement that all derivations into a certain class of modules are inner, by the weaker requirement that they are semi-inner (a difference of two module multiplications). 
This relaxation turns out to be robust (closed under many standard constructions) and yet non-trivial.

We adopt standard notation: $\A$ is a Banach algebra, $\A^*$ its dual, and $\Delta(\A)$ denotes the character space (non-zero multiplicative linear functionals on $\A$). 
All modules are Banach $\A$-bimodules and derivations are continuous unless stated otherwise. 

A Banach algebra $\mathcal{A}$ is called $\varphi$-amenable, for $\varphi\in\Delta(\mathcal{A})$, if there exists a bounded linear functional $m$ on $\mathcal{A}^*$ such that $m(\varphi)=1$ and $m(f\cdot a)=\varphi(a) m(f)$ for each $a\in \mathcal{A}$ and $f\in \mathcal{A}^*$.
 
Let $X$ be a Banach $\mathcal{A}$-bimodule. 
A continuous linear map $D:\A\to X$ is a derivation if
\[
D(ab)=a\cdot D(b)+D(a)\cdot b\qquad(a,b\in\A).
\]
Given $x\in X$, the mapping $\mathrm{ad}_x:\mathcal{A}\to X$ defined by $\mathrm{ad}_x(a)=a \cdot x - x \cdot a$ is a derivation which called inner derivation associated with $x$.
Let $Z^1(\mathcal{A},X)$ denote the space of all continuous derivations and $N^1(\mathcal{A},X)$ the space of all inner derivations from $\mathcal{A}$ into $X$.  
The first cohomology group $H^1(\mathcal{A},X)$ is the quotient $Z^1(\mathcal{A},X)/N^1(\mathcal{A},X)$. 

We also note that if $X$ is a Banach $\mathcal{A}$-bimodule, then the dual $X^*$ of $X$ has a natural Banach $\mathcal{A}$-bimodule structure given by 
$$  (a\cdot f)(x) =  f(x\cdot a ), \ \ \ \ \   (f\cdot a)(x) =  f(a\cdot x) \ \ \ \ \ (a\in\mathcal{A}, x\in X, f\in X^*).$$
In \cite{KLP1}, $\varphi$-amenability was characterized in several different ways. 
In fact, $\mathcal{A}$ is $\varphi$-amenable if and only if for every Banach $\mathcal{A}$-bimodule $X$ such that $a\cdot x=\varphi(a) x$ for $a\in\mathcal{A}, x\in X$, we have $H^1(\mathcal{A},X)=\{0\}$ \cite[Theorem 1.1]{KLP1}.
In other words, $\mathcal{A}$ is $\varphi$-amenable if and only if there exists a bounded net $(u_\alpha)$ in $\mathcal{A}$ 
such that $\| au_\alpha-\varphi(a) u_\alpha \|\to 0$ for all $a\in \mathcal{A}$, and $\varphi(u_\alpha)=1$ for all $\alpha$ \cite[Theorem 1.4]{KLP1}.

Another notion related to this amenability is $\varphi$-contractibility of $\mathcal{A}$ for $\varphi\in\Delta(\mathcal{A})$ which was introduced and studied by Hu et al. \cite{HMT}.
In fact, $\mathcal{A}$ is $\varphi$-contractible if there exists an element $m$ in the projective tensor product $\mathcal{A}\hat{\otimes}\mathcal{A}$ such that
$$\varphi(\pi(m))=1 \ \ \ \ \ \  and \ \ \ \ \ \ a\cdot m=\varphi(a) m,$$
for any $a\in\mathcal{A}$. We note that $\pi$ is the usual products morphism from $\mathcal{A}\hat{\otimes}\mathcal{A}$ into $\mathcal{A}$ given by $\pi(a\otimes a)=ab$ for all $a,b\in\mathcal{A} $.
For more information about $\varphi$-amenability and related concepts we refer the readers 
 to \cite{KLP1, KLP2, HMT, ASW, LS, M, SM}. Also, the first author has recently studied the essential character contractibility of Banach algebras \cite{S}.
 
 Recent developments in the theory of amenability-type properties of Banach algebras have focused on generalized notions such as character amenability, module amenability, and various approximate and pseudo versions of amenability (see \cite{DS, TR, STB}). These developments have led to new insights into the homological properties of Banach algebras and their tensor products, dual structures, and module actions. The present work contributes to this line of research by introducing and studying the notion of $\varphi$-semi contractibility and its approximate counterpart. 

Our work is organized as follows. In Section 2, we present the main definitions and establish a fundamental characterization of $\varphi$-semi contractibility and  explore its relationship with other amenability-type properties. In Section 3, investigate hereditary properties, including behavior under projective tensor products, quotients, and closed ideals. In Section 4, we introduce the approximate version of $\varphi$-semi contractibility and prove a corresponding characterization theorem. Moreover, we prove  $\mathcal{A} \widehat{\otimes} \mathcal{B}$ is approximately $\varphi \otimes \psi$-semi contractible if and only if $\mathcal{A}$ is approximately $\varphi$-semi contractible and $\mathcal{B}$ is approximately $\psi$-semi contractible.

\section{Foundations of $\varphi$-Semi Contractibility}
The concept of semi-inner derivations provides a natural bridge between inner derivations and arbitrary derivations. Recall from \cite{GL} that a derivation $D: \mathcal{A} \to X$ is called semi-inner if there exist $x_1, x_2 \in X$ such that
$D=\mathrm{ad}_{x_1,x_2}$, i.e.  $D(a) = a \cdot x_1 - x_2 \cdot a$ for all $a \in \mathcal{A}$. This notion motivates our central definition.

\begin{definition}
Let $\mathcal{A}$ be a Banach algebra and $\varphi \in \Delta(\mathcal{A})$. We say that $\mathcal{A}$ is \emph{$\varphi$-semi contractible} if every derivation $D : \mathcal{A} \to X$ into any Banach $\mathcal{A}$-bimodule $X$ with right module action given by $x \cdot a = \varphi(a)x$ is semi-inner.
\end{definition}

This definition relaxes the requirement of $\varphi$-contractibility by allowing two elements to witness the "almost inner" nature of derivations, rather than requiring a single element that simultaneously serves both roles. 
The following theorem states a condition equivalent to the $\varphi$-semi contractibility of the Banach  algebra $\mathcal{A}$.

\begin{theorem}\label{22}
\rm Let $\mathcal{A}$ be a Banach algebra and $\varphi\in\Delta(\mathcal{A})$. Then the following statements are equivalent:
\begin{enumerate}
\item $\mathcal{A}$ is $\varphi$-semi contractible.\label{t1:1}
\item There exist $m_1, m_2 \in \mathcal{A}$ such that $\varphi(m_1)=\varphi(m_2)=1$ and $am_1=\varphi(a) m_2$ for all $a\in \mathcal{A}$.\label{t1:2}
\end{enumerate}
\end{theorem} 

\begin{proof}
(\ref{t1:1})$\Rightarrow $(\ref{t1:2}) Suppose that $\mathcal{A}$ is $\varphi$-semi contractible. We first note that $\ker\varphi$ is a closed ideal in $\mathcal{A}$ of codimension 1 and in fact a Banach $\mathcal{A}$-bimodule through the following actions
$$a\cdot x=ax,\ \ x\cdot a=\varphi(a)x\ \ (a\in \mathcal{A}, x\in\ker\varphi).$$
Let $u\in \mathcal{A}$ be such that $\varphi(u)=1$. Define the map $D:\mathcal{A}\to \ker\varphi$ by $D(a)=au-\varphi(a)u$. In this case
\begin{align*}
a\cdot D(b)+ D(a)\cdot b &= a\cdot (bu-\varphi(b)u)+(au-\varphi(a)u)\cdot b\\
&= abu-\varphi(b)au+\varphi(b)au-\varphi(b)\varphi(a)u\\
&= abu-\varphi(ab)u=D(ab),
\end{align*}
so $D$ is a derivation and therefore there are $x_1,x_2\in\ker\varphi$ such that $D=ad_{x_1,x_2}$. Now let $m_1=u-x_1$ and $m_2=u-x_2$.  We have $\varphi(m_1)=\varphi(m_2)=1$ and for any $a\in\mathcal{A}$,
\begin{align*}
am_1-\varphi(a)m_2 &= au-ax_1-(\varphi(a)u-\varphi(a)x_2)\\
&= au-\varphi(a)u-(a\cdot x_1-x_2\cdot a)\\
&= D(a)-D(a)=0
\end{align*}
and so $am_1=\varphi(a)m_2$.

(\ref{t1:2})$\Rightarrow$(\ref{t1:1}) Suppose that $D:\mathcal{A}\to X$ is a derivation for some Banach $\mathcal{A}$-bimodule $X$ with right action $x\cdot a=\varphi(a) x$. For any $a\in \mathcal{A}$ we have
\begin{align*}
a\cdot D(m_1)-D(m_2)\cdot a &= D(a\cdot m_1)-D(a)\cdot m_1-D(m_2)\cdot a\\
&= \varphi(a) D(m_2)-D(a)-\varphi(a) D(m_2)\\
&= -D(a).
\end{align*}
It follows that $D(a)=a\cdot(-D(m_1))-(-D(m_2))\cdot a$ or equivalently\linebreak  $D=\mathrm{ad}_{(-D(m_1),-D(m_2))}$ and thus $\mathcal{A}$ is $\varphi$-semi contractible by definition.
\end{proof}

This characterization provides a concrete algebraic criterion for $\varphi$-semi contractibility that is often easier to verify than the original definition.

\begin{example}
Let $A$ be the vector space of all complex $3 \times 3$ matrices of the form
$$ a =
\begin{pmatrix}
0 & x & 0 \\
0 & 0 & 0 \\
0 & 0 & \lambda
\end{pmatrix},
\qquad x, \lambda \in \mathbb{C}.
$$
Clearly, $A$ with the norm $\|a\| = |x| + |\lambda|$ is a Banach algebra. We define the following map
$$ \varphi : A \to \mathbb{C}, \qquad 
\varphi\!\left(
\begin{pmatrix}
0 & x & 0 \\
0 & 0 & 0 \\
0 & 0 & \lambda
\end{pmatrix}
\right)
= \lambda,
$$
for $x, \lambda \in \mathbb{C}$. Note that $\varphi$ is a character on $A$.
Let
$$ m_1 =
\begin{pmatrix}
0 & 1 & 0 \\
0 & 0 & 0 \\
0 & 0 & 1
\end{pmatrix},
\qquad
m_2 =
\begin{pmatrix}
0 & 0 & 0 \\
0 & 0 & 0 \\
0 & 0 & 1
\end{pmatrix}.
$$
Then $\varphi(m_1) = \varphi(m_2) = 1$, and for every
$ a = \begin{pmatrix}
0 & x & 0 \\
0 & 0 & 0 \\
0 & 0 & \lambda
\end{pmatrix} $ in $A$, we have
$$
a m_1 =
\begin{pmatrix}
0 & 0 & \lambda \\
0 & 0 & 0 \\
0 & 0 & \lambda
\end{pmatrix}
= \lambda
\begin{pmatrix}
0 & 0 & 1 \\
0 & 0 & 0 \\
0 & 0 & 1
\end{pmatrix}
= \varphi(a)m_2.
$$
So $A$ is $\varphi$-semi contractible.
\end{example}

\begin{example}
let $\mathcal{A}=\mathbb{C}^2$ with multiplication $(a,b)(c,d)=(ac,0), (a,b,c,d\in \mathbb{C})$ and the norm $\|(a,b)\|=|a|+|b|$. Define the character $\varphi(a,b)=a, (a, b\in \mathbb{C})$. Choose distinct elements $m_1=(1,2)$ and $m_2=(1,0)$. Then $\varphi(m_1)=\varphi(m_2)=1$. Now for any $(a,b)\in \mathcal{A}$ we have
$(a,b)m_1=(a,0)=\varphi(a,b)m_2.$ So $\mathcal{A}$ is $\varphi$-semi contractible.

\end{example}

To better understand the position of $\varphi$-semi contractibility in the hierarchy of amenability-type properties, we introduce and relate it to other relevant concepts.

Let $ \mathcal{A} $ be a Banach algebra. An \( \mathcal{A} \)-bimodule \( X \) is called dual if there is a closed submodule \( X_* \) of \( X^* \) such that $X = (X_*)^*$.   
A Banach algebra \( \mathcal{A} \) is called dual if it is dual as a Banach \( \mathcal{A} \)-bimodule. 
A dual Banach \( \mathcal{A} \)-bimodule \( X \) is normal if for each \( x \in X \) the module maps  
$
\mathcal{A} \to X$;  $a \mapsto a \cdot x$ and  $a \mapsto x \cdot a $  
are  weak$^*$-continuous.
Furthermore, a Banach \( \mathcal{A} \)-bimodule \( X \) is neo-unital if  $X = \mathcal{A} \cdot X \cdot \mathcal{A}$.  

\begin{definition}
Let $\mathcal{A}$ be a Banach algebra and $\varphi \in \Delta(\mathcal{A})$.
\begin{enumerate}
\item $\mathcal{A}$ is called $\varphi$-semi amenable  if  for every Banach $\mathcal{A}$-bimodule $X$ with the left  module action $a \cdot x = \varphi(a)x,$ $(a\in \mathcal{A}, x\in X)$ every  derivation $ D: \mathcal{A} \to X^* $ is semi-inner.

\item If $\mathcal{A}$ is a dual Banach algebra and $\varphi \in \Delta_{w^*}(\mathcal{A})$ (the weak$^*$-continuous characters), then $\mathcal{A}$ is called $\varphi$-semi Connes amenable if for every normal dual Banach $\mathcal{A}$-bimodule $X=(X_*)^*$ with neo-unital $X_*$ and the right module action $x \cdot a = \varphi(a)x$, $(a\in \mathcal{A}, x\in X)$,
every weak$^*$-continuous derivation $ D: \mathcal{A} \to X $ is semi-inner.

\end{enumerate}
\end{definition}

The following result clarifies the relationship between these concepts for dual Banach algebras.

\begin{theorem}
Suppose that $\mathcal{A}$ is a dual Banach algebra and $\varphi \in \Delta_{w^*}(\mathcal{A})$.
Then $\mathcal{A}$ is $\varphi$-semi contractible if and only if $\mathcal{A}$ is $\varphi$-semi Connes amenable.
\end{theorem}

\begin{proof}
Let $\mathcal{A}$ be a $\varphi$-semi contractible Banach algebra.
By definition of $\varphi$-semi amenability, $\mathcal{A}$ is $\varphi$-semi amenable.
Let $D : \mathcal{A} \to X$ be a weak$^*$-continuous derivation for some normal dual Banach $\mathcal{A}$-bimodule $X$ with 
$x \cdot a = \varphi(a)x$.
Clearly $D$ is bounded and so $D$ is semi-inner.
Thus $\mathcal{A}$ is $\varphi$-semi Connes amenable.

Conversely, suppose that $\mathcal{A}$ is $\varphi$-semi Connes amenable.
Obviously, $\ker \varphi$ with the right action $x \cdot a = \varphi(a)x$ and the natural left action 
is a normal dual Banach $\mathcal{A}$-bimodule.
Choose $e \in \mathcal{A}$ such that $\varphi(e) = 1$ and define 
$$ D(a) = ae - \varphi(a)e. $$
Then $D$ is a weak$^*$-continuous derivation from $\mathcal{A}$ into $\ker \varphi$.
By assumption, there are $x_1$ and $x_2$ in $\ker \varphi$ such that 
$ D = \mathrm{ad}_{x_1, x_2}$. 
Now let $m_1 = e - x_1$ and $m_2 = e - x_2$.
By a similar argument as in the proof of Theorem \ref{22}, we can show that
$ a m_1 = \varphi(a) m_2 \quad (a \in \mathcal{A}) $ 
and 
$ \varphi(m_1) = \varphi(m_2) = 1$. 
Therefore, $\mathcal{A}$ is $\varphi$-semi contractible.
\end{proof}

We now examine when semi-inner derivations become inner, which reveals connections with approximate identities and module structures.  
We recall from \cite{D} that a net $(e_\alpha)$ in $\mathcal{A}$ is a left  approximate identity for $X$ if $(e_\alpha)$ is a left approximate identity for $\mathcal{A}$ and also 
$\lim_\alpha  e_\alpha\cdot x =x$ $(x\in X)$.

\begin{proposition}\label{pro23}
Let $\mathcal{A}$ be a Banach algebra, $X$  a Banach $\mathcal{A}$-bimodule with right
  action $x\cdot a=\varphi(a)x$,  and 
$D:\mathcal{A}\to X^*$  a semi-inner derivation.  Then  $D$ is inner in each of the following cases: 
\begin{enumerate}
\item $\mathcal{A}$ has a left approximate identity for $X$.\label{p1:1}
\item $X$ is neo-unital.\label{p1:2}
\item $\mathcal{A}$ has a bounded approximate identity.\label{p1:3}
\end{enumerate}
\end{proposition}

\begin{proof}
First we note that for any $a\in\mathcal{A}$ and $m\in X^*$ we have $a\cdot m=\varphi(a)m$.  Since $D$ is semi-inner, there are $m,n\in X^*$ such that for every $a\in\mathcal{A}$, $D(a)=a\cdot n-m\cdot a$. 
According to $D(ab)=a\cdot D(b)+D(a)\cdot b$, for $a,b\in\mathcal{A}$, we get  
$$\varphi(ab)n - m \cdot ab = \varphi(ab)n - a\cdot m\cdot b + \varphi(a)n\cdot b - m \cdot ab$$
and so $\varphi(a)(n-m)\cdot b = 0$. 

(\ref{p1:1})  Let $(e_\alpha)\in \mathcal{A}$ be a left approximate identity For $X$. So for every $a\in \mathcal{A}$,
\begin{align*}
D(a) &= a\cdot n-m\cdot a\\
&= \lim_\alpha \varphi(e_\alpha)(a\cdot n-m\cdot a)=\lim_\alpha \varphi(e_\alpha)(a\cdot n-n\cdot a+n\cdot a -m\cdot a)\\
&= \lim_\alpha (\varphi(e_\alpha)(a\cdot n-n\cdot a)+\varphi(e_\alpha)(n-m)\cdot a )\\
&= a\cdot n-n\cdot a.
\end{align*}
 That is D is inner.
 
(\ref{p1:2}) Let $x$ be an arbitrary element of $X$. By definition there exists $a_1,a_2 \in\mathcal{A}$ and $x'\in X$ such that 
$x=a_1\cdot x'\cdot a_2$. This and the fact that $\varphi(a)(n-m)\cdot b = 0$, for $a,b\in \mathcal{A}$, implies that
$$ \langle (n-m), x \rangle = \langle (n-m), a_1\cdot  x'\cdot a_2 \rangle = \langle \varphi(a_2)(n-m)\cdot a_1 , x' \rangle = 0. $$
This means that $n-m=0$. Thus $n=m$ and so $D$ is inner.
 
(\ref{p1:3}) Let $(e_\alpha)$ be  a bounded approximate identity of $\mathcal{A}$. 
Suppose that $L$ is a limit of the left multiplication operators $L_{e_\alpha}$ on $X^*$
 by the elements $e_\alpha$ in the weak$^*$ operator topology, and $R$ similarly for right multiplication.
By a similar argument of Proposition 3.3 of \cite{GL}, we have the following decomposition 
$$X^*=LRX^*\oplus L(I-R)X^*\oplus (I-L)X^*.$$
Now set 
$$D_1=LRD,\ D_2=L(I-R)D, \ D_3=(I-L)D.$$
It is easy to see that these are derivations into the corresponding summands of the above decomposition. 
One can show that $D_2$ and $D_3$ are inner, and 
\begin{align*}
D_1(a)&= LR(a\cdot n-m\cdot a) \\
&= \textrm{w}^*-\lim_\alpha \lim_\beta (e_\alpha  a\cdot n\cdot e_\beta - e_\alpha\cdot m\cdot a e_\beta ).
\end{align*}
But from the fact that $\varphi(a)(n-m)\cdot b=0$, it follows that
$$
e_\alpha a \cdot n \cdot e_\beta = \varphi(e_\alpha a) m\cdot e_\beta = e_\alpha a\cdot  m\cdot e_\beta.
$$
So $D_1(a)=a\cdot LR(m)-LR(m)\cdot a$, for each $a\in\mathcal{A} $. That is $D_1$ is inner and thus  $D$ is inner.
\end{proof}

The notion of essential $\varphi$-amenability provides another important variant that restricts attention to neo-unital modules, which often arise naturally in practice.

\begin{definition}\label{**}
Let \( \mathcal{A} \) be a Banach algebra and $ \varphi \in \Delta(\mathcal{A}) $. We say that \( \mathcal{A} \) is essentially \(\varphi\)-amenable (respectively, essentially \(\varphi\)-semi amenable) if for every neo-unital Banach \( \mathcal{A} \)-bimodule \( X \) with the right module action $ x\cdot a = \varphi(a)x  (a\in \mathcal{A}, x\in X)$, every continuous derivation \( D: \mathcal{A} \rightarrow X^* \) is inner (respectively, is semi-inner).
\end{definition}

The following result is now immediate from Definition \ref{**} and Proposition \ref{pro23}.

\begin{theorem}
Let $\mathcal{A}$ be a Banach algebra and \( \varphi \in \Delta(\mathcal{A}) \). Then the following statements hold:
\begin{enumerate}
\item $\mathcal{A}$ is essentially \(\varphi\)-amenable if and only if $\mathcal{A}$ is essentially \(\varphi\)-semi amenable.
\item If $\mathcal{A}$ has a bounded approximate identity, then $\mathcal{A}$ is \(\varphi\)-left amenable if and only if $\mathcal{A}$ is \(\varphi\)-semi amenable.
\end{enumerate}
\end{theorem}

\section{Hereditary Properties and Tensor Products}

A fundamental aspect of any algebraic property is its behavior under standard constructions. We begin by examining how $\varphi$-semi contractibility behaves with respect to ideals and quotient algebras.

\begin{proposition}
Let $\mathcal{A}$ be a Banach algebra, $I$ a closed two-sided ideal of $\mathcal{A}$, and $\varphi \in \Delta(\mathcal{A})$ such that $I \subseteq \ker \varphi$. If $\mathcal{A}$ is $\varphi$-semi contractible, then $\mathcal{A}/I$ is $\widetilde{\varphi}$-semi contractible, where $\widetilde{\varphi}$ is the character induced by $\varphi$ on $\mathcal{A}/I$.
\end{proposition}
 
\begin{proof}
It is clear.
\end{proof}

The converse direction requires additional hypotheses:

\begin{proposition}
Let $\mathcal{A}$ be a Banach algebra, $I$ a closed two-sided ideal of $\mathcal{A}$, and $\varphi \in \Delta(\mathcal{A})$ with $I \subseteq \ker \varphi$. If $I$ has a right identity and $\mathcal{A}/I$ is $\widetilde{\varphi}$-semi contractible, then $\mathcal{A}$ is $\varphi$-semi contractible.
\end{proposition} 

\begin{proof}
By assumption there are $n_1+I$ and $n_2+I$ in ${\mathcal{A}}/{I}$ such that
 $\widetilde{\varphi}(n_1+I)=\widetilde{\varphi}(n_2+I)=1$, and for any $a\in\mathcal{A}$,  
 $(a+I)(n_1+I)=\widetilde{\varphi}(a+I)(n_2+I)$.
 
Define $m_1=n_1-n_1e$ and $m_2=n_2-n_2e$, where $e$ is a right identity for $I$. From the fact that 
$\varphi(n_1)=\varphi(n_2)=1$ and $n_1e,n_2e\in I$, it follows that $\varphi(m_1)=\varphi(m_2)=1$ and 
for every $a\in \mathcal{A}$, we have 
\begin{align*}
am_1-\varphi(a) m_2 &= an_1-an_1 e - \varphi(a) n_2+\varphi(a) n_2 e\\
&= an_1-\varphi(a)n_2-(an_1-\varphi(a) n_2)e=0
\end{align*}
since $an_1 - \varphi(a)n_2\in I$. It follows that $am_1=\varphi(a)m_2$ and so $\mathcal{A}$ is $\varphi$-semi contractible.
\end{proof}

When the ideal does not lie entirely in the kernel of $\varphi$, we have a different inheritance pattern:

\begin{proposition} 
Let $\mathcal{A}$ be a Banach algebra, $I$  a closed two-sided ideal of $\mathcal{A}$, and $\varphi\in\Delta(\mathcal{A})$ with $\varphi_{|_I}\neq 0$.
Then $\mathcal{A}$ is $\varphi$-semi contractible if and only if $I$ is $\varphi_{|_I}$-semi contractible.
\end{proposition} 

\begin{proof}
Suppose that $\mathcal{A}$ is $\varphi$-semi contractible. Then there are $m_1$ and $m_2$ in $\mathcal{A}$ such that $\varphi(m_1)=\varphi(m_2)=1$ and $am_1=\varphi(a)m_2$, for each $a\in A$. 
Choose $b_0\in I$ such that $\varphi(b_0)=1$ and define $m_1':= b_0m_1$ and $m_2':=b_0m_2$. 
Then $\varphi_{|_I}(m_1')=\varphi_{|_I}(m_2')=1$ and for every $b\in I$ we have
$$b m_1'=bb_0m_1=\varphi(bb_0)m_2=\varphi(b_0b)m_2=b_0bm_2=\varphi(b)b_0m_2=\varphi(b)m_2',$$
and therefore $I$ is $\varphi_{|_I}$-semi contractible.

For the converse, suppose that $I$ is $\varphi_{|_I}$-semi contractible. Then there are $m_1',m_2'\in I$ such that $\varphi_{|_I}(m_1')=\varphi_{|_I}(m_2')=1$ and $bm_1'=\varphi(b)m_2'$ for each $b\in I$. 
Choose $b_0\in I$  such that $\varphi(b_0)=1$ and define $m_1=b_0m_1'$ and $m_2=b_0m_2'$.
Thus $\varphi(m_1)=\varphi(m_2)=1$ and for every $a\in \mathcal{A}$ we have
\begin{align*}
am_1=ab_0m_1' &=\varphi(ab_0)m_2' =\varphi(a)\varphi(b_0b_0)m_2' \\
& =\varphi(a) (b_0b_0m_1')=\varphi(a) (b_0\varphi(b_0) m_2')=\varphi(a) (b_0m_2')=\varphi(a)m_2,
\end{align*}
so $\mathcal{A}$ is $\varphi$-semi contractible.
\end{proof}

The next result shows that $\varphi$-semi contractibility is preserved under continuous homomorphisms with dense range:

\begin{proposition}
Let $\mathcal{A}$ and $\mathcal{B}$ be Banach algebras, $\psi\in\Delta(\mathcal{B})$, and $h:\mathcal{A}\to \mathcal{B}$ be a continuous homomorphism with dense range. If $\mathcal{A}$ is $\psi\circ h$-semi contractible, then $\mathcal{B}$ is $\psi$-semi contractible.
\end{proposition}
 
\begin{proof}
By assumption, there are $m_1,m_2\in \mathcal{A}$ such that $\psi\circ h(m_1)=\psi\circ h(m_2)$ and for any $a\in\mathcal{A}$,  $am_1=\psi\circ h(a) m_2$. 
Set $m_1^{\mathcal{B}}=h(m_1)$ and $m_2^{\mathcal{B}}=h(m_2)$. 
Clearly $\psi(m_1^{\mathcal{B}})=\psi(m_2^{\mathcal{B}})=1$. 
Now let $b\in \mathcal{B}$ and  $(a_\alpha)_\alpha$ be a net in $\mathcal{A}$ such that $h(a_\alpha)\to b$. Then
\begin{align*}
b m_1^{\mathcal{B}}=bh(m_1)=\lim_\alpha h(a_\alpha)h(m_1) &= \lim_\alpha h(a_\alpha m_1)=\lim_\alpha h(\psi\circ h(a_\alpha)m_2) \\
&= \lim_\alpha \psi\circ h(a_\alpha)h(m_2) \\
&= \psi(b)m_2^{\mathcal{B}},
\end{align*}
therefore $\mathcal{B}$ is $\psi$-semi contractible.
\end{proof}

The following result demonstrates that essential $\varphi$-semi amenability can be \emph{lifted} from a sufficiently large subalgebra to the whole algebra, providing an important property that flows from substructures to larger structures.

\begin{proposition}
Let $\mathcal{A}$ be a Banach algebra and $\varphi \in \Delta(\mathcal{A})$. Let $\mathcal{B}$ be a subalgebra of $\mathcal{A}$ containing $\mathcal{A}^2$. 
If $\mathcal{B}$ is essentially $\varphi_{|_\mathcal{B}}$-semi amenable, then $\mathcal{A}$ is essentially $\varphi$-semi amenable.
\end{proposition}

\begin{proof}
Suppose that $\mathcal{B}$ is essentially $\varphi_{|_\mathcal{B}}$-semi amenable.  
Let $D: \mathcal{A} \to X^*$ be a derivation for some neo-unital $\mathcal{A}$--bimodule $X$ with the right module action 
$x \cdot a = \varphi(a)x$ for all $a \in \mathcal{A}$ and  $x \in X$.

Moreover, $X = \mathcal{A} \cdot X = \mathcal{A} \cdot (A \cdot X) \subseteq \mathcal{B} \cdot X \subseteq X$, and so $\mathcal{B} \cdot X = X$. 
Similarly, $X \cdot \mathcal{B} \subseteq X$. Thus, $X$ is a neo-unital Banach $\mathcal{B}$--bimodule.

Since $D_{|_\mathcal{B}} : \mathcal{B} \to X^*$ is a derivation, there exist $x_1^*, x_2^* \in X^*$ such that
$$ D_{|_\mathcal{B}} = \operatorname{ad}_{x_1^*, x_2^*}.$$
So, for every $b_1, b_2 \in \mathcal{B}$, we have
$$ \varphi(b_1)(x_1^* - x_2^*) \cdot b_2 = 0. $$
Let $x$ be an arbitrary element of $X$. By definition, there exist $b_1, b_2 \in \mathcal{B}$ and $x' \in X$ such that
$$ x = b_1 \cdot x' \cdot b_2. $$
Hence, for every $a_1, a_2 \in \mathcal{A}$, we have
\begin{align*}
\langle a_1, (x_2^* - x_1^*) \cdot a_2, x \rangle  
& = \langle a_1, (x_2^* - x_1^*) \cdot a_2, b_1 \cdot x' \cdot b_2 \rangle \\
& = \langle b_2 a_1 \cdot  (x_2^* - x_1^*) \cdot a_2 b_1, x' \rangle \\
& = \langle \varphi(b_2 a_1)  (x_2^* - x_1^*) \cdot a_2 b_1, x' \rangle  \\
& = 0.
\end{align*}
This means that 
$$ a_1 \cdot (x_2^* - x_1^*) \cdot a_2 = 0 \quad \text{for all } a_1, a_2 \in \mathcal{A}. $$
Define 
$$ \overline{D} = D - \operatorname{ad}_{x_1^*, x_2^*}. $$
From the fact that $a_1 \cdot (x_2^* - x_1^*) \cdot a_2 = 0$ for all $a_1, a_2 \in \mathcal{A}$, 
we can show that $\overline{D}$ is a derivation from $A$ into $X$ and $\overline{D}_{|_\mathcal{B}} = 0$.

Let $a \in \mathcal{A}$ and $x \in X$. Since $X$ is neo-unital, there exist $x_1 \in X$ and $b_1 \in \mathcal{B}$ such that 
$x = b_1 \cdot x_1$. Consequently,
\begin{align*}
\langle \overline{D}(a), x \rangle 
= \langle \overline{D}(a), b_1 \cdot x_1 \rangle 
= \langle \overline{D}(a) \cdot b_1, x_1 \rangle 
= \langle \overline{D}(a b_1) - a \cdot \overline{D}(b_1), x_1 \rangle = 0.
\end{align*}
So $\overline{D} = 0$, and hence $D = \operatorname{ad}_{x_1^*, x_2^*}$.  

Therefore, $\mathcal{A}$ is essentially $\varphi$-semi amenable.
\end{proof}

The behavior of $\varphi$-semi contractibility under projective tensor products is particularly important, as it reveals how the property interacts with one of the fundamental constructions in Banach algebra theory.

\begin{theorem}
Let $\mathcal{A}$ and $\mathcal{B}$ be Banach algebras with $\varphi\in\Delta(\mathcal{A})$ and $\psi\in\Delta(\mathcal{B})$. Then $\mathcal{A}\widehat{\otimes} \mathcal{B}$ is 
$\varphi\otimes\psi$-semi contractible if and only if $\mathcal{A}$ is $\varphi$-semi contractible  and $\mathcal{B}$ is $\psi$-semi contractible.
\end{theorem} 

\begin{proof}
Suppose that $\mathcal{A}\widehat{\otimes} \mathcal{B}$ is $\varphi\otimes\psi$-semi contractible. 
Then there exist $m_1,m_2\in \mathcal{A}\widehat{\otimes} \mathcal{B}$  with $\varphi\otimes\psi(m_1)=\varphi\otimes\psi(m_2)=1$ and 
$a\otimes b\cdot m_1=\varphi\otimes\psi(a\otimes b)m_2$, for all $a\in \mathcal{A}$ and  $b\in \mathcal{B}$. 

Define the linear map $\Upsilon : \mathcal{A}\widehat{\otimes} \mathcal{B}\to \mathcal{A}$ 
by $\Upsilon(a\otimes b) = a\psi(b)$ as in the proof of  Theorem 3-14 in \cite{N}. 
Let $m^{\mathcal{A}}_1=\Upsilon(m_1)$ and $m^{\mathcal{A}}_2=\Upsilon(m_2)$. 
Since $\varphi\circ\Upsilon = \varphi\otimes\psi$, one can easily see that $\varphi(m^{\mathcal{A}}_1)=\varphi(m^{\mathcal{A}}_2)=1$. Choose $a_0\in \mathcal{A}$ and $b_0\in \mathcal{B}$ such that $\varphi(a_0)=1$ and $\psi(b_0)=1$. 
Thus For any $a\in \mathcal{A}$
\begin{align*}
a m^{\mathcal{A}}_1&=a\Upsilon(m_1)\\
&=a\Upsilon((a_0\otimes b_0)m_1)
=\Upsilon(a a_0\otimes b_0\cdot m_1)=\Upsilon(\varphi(a)m_2) = 
\varphi(a)m^{\mathcal{A}}_2.
\end{align*}
Therefore $\mathcal{A}$ is $\varphi$-semi contractible. The same conclusion can be drawn for $\psi$-semi contractibility of $\mathcal{B}$.

Conversely, assume that $\mathcal{A}$ is $\varphi$-semi contractible and $\mathcal{B}$ is $\psi$-semi contractible. 
Then there are $m^{\mathcal{A}}_1,m^{\mathcal{A}}_2\in \mathcal{A}$ and $m^{\mathcal{B}}_1,m^{\mathcal{B}}_2\in \mathcal{B}$ such that
\begin{align*}
\varphi(m_1^\mathcal{A}) &= \varphi(m_2^\mathcal{A}) = 1, \quad a m_1^\mathcal{A} = \varphi(a) m_2^\mathcal{A} \quad (a \in \mathcal{A}), \\
\psi(m_1^\mathcal{B}) &= \psi(m_2^\mathcal{B}) = 1, \quad b m_1^\mathcal{B} = \psi(b) m_2^\mathcal{B} \quad (b \in \mathcal{B}).
\end{align*}
Define $m_1=m^{\mathcal{A}}_1\otimes m^{\mathcal{B}}_1$ and 
$m_2=m^{\mathcal{A}}_2\otimes m^{\mathcal{B}}_2$. 
So for every $a\in \mathcal{A}$ and $b\in \mathcal{B}$ we have
\begin{align*}
(a\otimes b)m_1 &= a\otimes b(m^{\mathcal{A}}_1\otimes m^{\mathcal{B}}_1)= a m^{\mathcal{A}}_1\otimes b m^{\mathcal{B}}_1\\
&=(\varphi(a),\psi(b))(m^{\mathcal{A}}_2\otimes m^{\mathcal{B}}_2)=\varphi\otimes \psi (a\times b)m_2.
\end{align*}
Also $\varphi\otimes\psi(m_1)=\varphi\otimes\psi(m_2)=1$. 
Therefore $\mathcal{A}\widehat{\otimes} \mathcal{B}$ is $\varphi\otimes \psi$-semi contractible.
\end{proof}

Let $\mathcal{A}$ and $\mathcal{B}$ be Banach algebras, it is well known that $\mathcal{A}\oplus_1 \mathcal{B}$, the $l^1$-direct sum of $\mathcal{A}$ and $\mathcal{B}$, is a Banach algebra with respect to the canonical multiplication defined by $ (a,b)(c,d):= (ac,bd) (a,c\in \mathcal{A}, b,d\in \mathcal{B})$ and note that  
$$\Delta(\mathcal{A}\oplus_1 \mathcal{B})=\{(\varphi,0)|\varphi\in \Delta(\mathcal{A})\}\cup\{(0,\phi)|\phi\in \Delta(\mathcal{B})\} $$

For brevity and clarity, we omit the proof of the following proposition.
\begin{proposition}
Let $\mathcal{A}$ and $\mathcal{B}$ be Banach algebras, $ \varphi \in \Delta(\mathcal{A})$ and $ \phi\in \Delta(\mathcal{B})$. Then the following statements hold:
\begin{enumerate}
\item \(\mathcal{A}\oplus_1 \mathcal{B} \) is \((\varphi,0)\)-semi contractible if and only if \( \mathcal{A} \) is \(\varphi\)-semi contractible.
\item \(\mathcal{A}\oplus_1 \mathcal{B} \) is \((0,\phi)\)-semi contractible if and only if \( \mathcal{B} \) is \(\phi\)-semi contractible.
\end{enumerate}
\end{proposition}

\section{Approximate $\varphi$-Semi Contractibility}

The investigation of approximate versions of algebraic properties has been a fruitful direction in Banach algebra theory, as evidenced by the study of approximate amenability and related concepts. In the context of $\varphi$-semi contractibility, it is natural to consider what happens when the exact algebraic conditions of Theorem \ref{22} are relaxed to hold only in an approximate sense. This leads us to the following definition of approximate $\varphi$-semi contractibility, which provides a quantitative measure of how closely an algebra approximates the $\varphi$-semi contractible condition.

\begin{definition}
A Banach algebra $\mathcal{A}$ is called approximately $\varphi$-semi contractible if for every Banach $\mathcal{A}$-bimodule $X$ with right action $x \cdot a = \varphi(a)x$, every continuous derivation $D: \mathcal{A} \to X$ is approximately semi-inner, i.e., there exist nets $(m_\alpha)$ and $(n_\alpha)$ in $X$ such that:
\[
D(a) = \lim_\alpha (a \cdot m_\alpha - n_\alpha \cdot a) = \lim_\alpha (a \cdot m_\alpha - \varphi(a) n_\alpha) \quad \text{for all } a \in \mathcal{A}.
\]
\end{definition}

The following characterization provides practical criteria for verifying approximate $\varphi$-semi contractibility. See Also \cite[Proposition 4.2]{ASW}.

\begin{proposition}
Let $\mathcal{A}$ be a Banach algebra and $\varphi\in\Delta(\mathcal{A})$. Then the following statements are equivalent:
\begin{enumerate}
\item $\mathcal{A}$ is approximately $\varphi$-semi contractible.\label{p2:1}
\item There exist nets $(m_\alpha)$ and $(n_\alpha)$ in $\mathcal{A}$ with $\varphi(m_\alpha)=\varphi(n_\alpha)=1$ and $\|am_\alpha-\varphi(a) n_\alpha\|\to 0$ for all $a\in \mathcal{A}$.\label{p2:2}
\item There exist nets $(m_\alpha)$ and $(n_\alpha)$ in $\mathcal{A}$ with $\varphi(m_\alpha) \to 1$, 
$\varphi(n_\alpha)\to 1$ and $\|am_\alpha-\varphi(a) n_\alpha\|\to 0$ for all $a\in \mathcal{A}$.\label{p2:3}
\end{enumerate}
\end{proposition}

\begin{proof}
(\ref{p2:1})$\Rightarrow$(\ref{p2:2}) Let  $\mathcal{A}$ be approximately $\varphi$-semi-contractible. Again we find that $\ker\varphi$ is a Banach $\mathcal{A}$-bimodule with the following actions
$$a\cdot x=ax,\ \ x\cdot a=\varphi(a)x\ \ (a\in \mathcal{A}, x\in\ker\varphi).$$
Let $u\in \mathcal{A}$ be such that $\varphi(u)=1$. Define $D:\mathcal{A}\to \ker\varphi$ by $D(a)=au-\varphi(a)u$. Clearly $D$ is a derivation. 
 Therefore there are nets $(x_\alpha)$ and $(y_\alpha)$ in $\ker\varphi$ such that
$$D(a)=\lim_\alpha(a\cdot x_\alpha-y_\alpha\cdot a)=\lim_\alpha (a\cdot x_\alpha-\varphi(a)y_\alpha),\ \ (a\in \mathcal{A}).$$
Now consider $m_\alpha=u-x_\alpha$ and $n_\alpha=u-y_\alpha$. Then $\varphi(m_\alpha)=\varphi(n_\alpha)=1$ and 
for each $a\in\mathcal{A}$
\begin{align*}
\|am_\alpha-\varphi(a)n_\alpha\|& = \|a(u-x_\alpha)-\varphi(a)(u-y_\alpha)\| \\
&= \|(au-\varphi(a)u)-(ax_\alpha-\varphi(a)y_\alpha\| \\
&= \| D(a)-(a\cdot x_\alpha-y_\alpha\cdot a) \|\to 0.
\end{align*}

(\ref{p2:2})$\Rightarrow$(\ref{p2:3}) Immediate.

(\ref{p2:3})$\Rightarrow$(\ref{p2:1}) Suppose that $D:\mathcal{A}\to X$ is a derivation for some Banach $\mathcal{A}$-bimodule $X$ with the right action $x\cdot a=\varphi(a) x$. For every $a\in \mathcal{A}$ we have 
\begin{align*}
\|D(a)  +a\cdot D(m_\alpha) &-\varphi(a) D(n_\alpha)\| \\
&= \|D(a)+D(a\cdot m_\alpha)-D(a)\cdot m_\alpha   -\varphi(a)D(n_\alpha)\| \\
&= \| D(a)+ D(a\cdot m_\alpha)-\varphi(m_\alpha) D(a) - D(\varphi(a) n_\alpha)\|\\
&\leq  \| D(a) - \varphi(m_\alpha) D(a) \| + \| D(a\cdot m_\alpha-\varphi(a) n_\alpha)\| \\
&\leq  \|D\| \| a-\varphi(m_\alpha) a\| + \| D\| \|a\cdot m_\alpha-\varphi(a) n_\alpha\| \to 0.
\end{align*}
It follows that $D(a)=\lim_\alpha (a\cdot (-D(m_\alpha)) - (-D(n_\alpha))\cdot a  )$ and so  $D$ is approximately 
semi-inner.
\end{proof}

The approximate version also behaves well with respect to tensor products:

\begin{theorem}
Let $\mathcal{A}$ and $\mathcal{B}$ be Banach algebras with $\varphi \in \Delta(\mathcal{A})$ and $\psi \in \Delta(\mathcal{B})$. Then $\mathcal{A} \widehat{\otimes} \mathcal{B}$ is approximately $\varphi \otimes \psi$-semi contractible if and only if $\mathcal{A}$ is approximately $\varphi$-semi contractible and $\mathcal{B}$ is approximately $\psi$-semi contractible.
\end{theorem}

\begin{proof}
Suppose that $\mathcal{A} \widehat{\otimes} \mathcal{B}$ is approximately $\varphi \otimes \psi$-semi contractible.  
Then there are nets $(M_\alpha)$ and $(N_\alpha)$ in $\mathcal{A} \widehat{\otimes} \mathcal{B}$ such that
$$ (\varphi \otimes \psi)(M_\alpha) = (\varphi \otimes \psi)(N_\alpha) = 1$$
and
$$ \| (a \otimes b) M_\alpha - (\varphi \otimes \psi)(a \otimes b) N_\alpha \| \to 0 $$
for all $a \in \mathcal{A}$ and $b \in \mathcal{B}$.

Define the linear map $\Upsilon : \mathcal{A} \widehat{\otimes} \mathcal{B} \to \mathcal{A}$ by 
$ \Upsilon(a \otimes b) = a \, \psi(b) $. Let $m_\alpha = \Upsilon (M_\alpha)$ and $n_\alpha = \Upsilon(N_\alpha)$.  
Since $\varphi \circ \Upsilon = \varphi \otimes \psi$, one can easily see that
$$ \varphi(m_\alpha) = \varphi(n_\alpha) = 1.$$
Choose $b_0 \in \mathcal{B}$ such that $\psi(b_0) = 1$.  
Thus, for every $a \in \mathcal{A}$ we have
\begin{align*}
\| a m_\alpha - \varphi(a) n_\alpha \| & = 
\| \Upsilon (a \otimes b_0) \Upsilon(M_\alpha) - \varphi \otimes \psi (a \otimes b_0) \Upsilon(N_\alpha) \| \\
 & = \| \Upsilon(a \otimes b_0\, M_\alpha) - \Upsilon((\varphi \otimes \psi) (a \otimes b_0) N_\alpha) \| \\
& \le \| \Upsilon \| \, \| (a \otimes b_0) M_\alpha - (\varphi \otimes \psi)(a \otimes b_0) N_\alpha \| \to 0. 
\end{align*}

Therefore, $\mathcal{A}$ is approximately $\varphi$--semi contractible.  
By a similar argument, we can show that $\mathcal{B}$ is approximately $\psi$--semi contractible.

Conversely, suppose that $\mathcal{A}$ is an approximately $\varphi$-semi contractible Banach algebra and $\mathcal{B}$ is an approximately $\psi$-semi contractible Banach algebra. 
Then there are nets $(m_\alpha^\mathcal{A})$ and $(n_\alpha^\mathcal{A})$ in $\mathcal{A}$ and $(m_\beta^\mathcal{B})$ and $(n_\beta^\mathcal{B})$ in $\mathcal{B}$ such that 
$$ \varphi(m_\alpha^\mathcal{A}) = \varphi(n_\alpha^{\mathcal{A}}) = 1, \quad \psi(m_\beta^\mathcal{B}) = \psi(n_\beta^{\mathcal{B}}) = 1, $$
and 
$$ \| a m_\alpha^\mathcal{A} - \varphi(a) n_\alpha^\mathcal{A} \| \to 0, \quad 
\| b m_\beta^\mathcal{B} - \psi(b)  n_\beta^\mathcal{B} \| \to 0. $$

Define 
$ M_{\alpha \beta } = m_\alpha^\mathcal{A} \otimes m_\beta^\mathcal{B}$ and $
N_{\alpha \beta } = n_\alpha^\mathcal{A} \otimes n_\beta^\mathcal{B}$. 
Then 
$ \varphi \otimes \psi (M_{\alpha \beta }) = \varphi \otimes \psi (N_{\alpha \beta }) = 1$,  
and for every $a \in \mathcal{A}$ and $b \in \mathcal{B}$ we have
\begin{align*}
\| a \otimes b \, M_{\alpha \beta }  - \varphi \widehat{\otimes}\psi (a\otimes b) N_{\alpha \beta } \| 
&  = \| ( a \otimes b) \, m_\alpha^\mathcal{A} \otimes m_\beta^\mathcal{B} - \varphi\widehat{\otimes}\psi (a\otimes b) n_\alpha^\mathcal{A} \otimes n_\beta^\mathcal{B} \| \\
& = \| a  m_\alpha^\mathcal{A} \otimes b m_\beta^\mathcal{B} - 
\varphi(a) n_\alpha^\mathcal{A} \otimes b m_\beta^\mathcal{B} + \\
&\quad   \varphi(a) n_\alpha^\mathcal{A} \otimes b m_\beta^\mathcal{B} - 
 \varphi(a) n_\alpha^\mathcal{A} \otimes \psi(b) n_\beta^\mathcal{B} \| \\
& \leq \| a m_\alpha^\mathcal{A} - \varphi(a) n_\alpha^\mathcal{A} \| \, \| b m_\beta^\mathcal{B} \| \\
& \quad + \| \varphi(a) n_\alpha^\mathcal{A} \| \, \| b m_\beta^\mathcal{B} - \psi(b) n_\beta^\mathcal{B} \| \to 0.
\end{align*}
Therefore, $\mathcal{A} \widehat{\otimes} \mathcal{B}$ is approximately $\varphi \otimes \psi$-semi contractible.
\end{proof}

\end{document}